\documentclass[11pt,hyp]{nyjm}
\usepackage{amsmath}
\usepackage{amssymb}
\usepackage{graphicx}
\usepackage{color}
\usepackage{upref}
\usepackage{url}
\usepackage{blkarray}
\usepackage{hyperref}
\hypersetup{nesting=true,debug=true,naturalnames=true}

\let\<\langle
\let\>\rangle
\newcommand{\doi}[1]{doi:\,\href{http://dx.doi.org/#1}{#1}}

\let\uml\"

\newtheorem{thm}{Theorem}
\newtheorem*{que*}{Question}
\newtheorem{prop}[thm]{Proposition}
\newtheorem{lem}[thm]{Lemma}

\newtheorem{remark}[thm]{Remark}
\newtheorem{que}[thm]{Question}

\newtheorem{example}[thm]{Example}

\begin{document}

\title[Minimal spectral radii of skew-reciprocal integer matrices]{On the minimal spectral radii of skew-reciprocal integer matrices}

\author{Livio Liechti}
\address{Department of Mathematics\\
University of Fribourg\\
Chemin du Mus\'ee 23\\
1700 Fribourg\\Switzerland}
\email{livio.liechti@unifr.fr}


\begin{abstract} 
We determine the minimal spectral radii among all skew-reciprocal 
integer matrices of a fixed even dimension that are primitive or nonnegative and irreducible. 
In particular, except for dimension six, we show that each such class of matrices 
realises smaller spectral radii than the corresponding reciprocal class.
\end{abstract}

\maketitle

\section{Introduction}
Curiously, orientation-reversing integer linear dynamical systems can be 
simpler than orientation-preserving ones in the following sense: among all
matrices~$A\in\mathrm{GL}_2(\mathbb{Z})$ with~$\det(A)=-1$, the smallest spectral 
radius $>1$ is the golden ratio~$\varphi$, while among matrices~$A$ with~$\det(A)=1$, 
the smallest spectral radius~$>1$ is~$\varphi^2$. In this article, we generalise this 
comparison to reciprocal and skew-reciprocal matrices of any even dimension, under the 
assumption of either primitivity or nonnegativity and irreducibility. 

A matrix is \emph{nonnegative} 
if all its coefficients are nonnegative. Such a matrix is \emph{primitive} if some power has strictly 
positive coefficients. A matrix is \emph{irreducible} if it is not conjugate via a permutation matrix 
to an upper triangular block matrix.
We call a matrix \emph{reciprocal} if the set of its eigenvalues (counted with multiplicity) 
is invariant under the transformation~$t\mapsto t^{-1}$. Finally, we call a matrix 
\emph{skew-reciprocal} if the set of its eigenvalues (counted with multiplicity) is invariant 
under the transformation~$t\mapsto -t^{-1}$. Important examples of reciprocal or skew-reciprocal 
matrices are symplectic or antisymplectic matrices, respectively.

We find out that with the exception of dimension six in the primitive case, the skew-reciprocal 
matrices always realise smaller spectral radii~$>1$ than the reciprocal ones. 

\begin{thm}
\label{prim_srwins}
Let~$g \ge 1$ and~$g \ne 3$. Among primitive matrices~$A \in \mathrm{GL}_{2g}(\mathbb{Z})$, the 
skew-reciprocal ones realise a smaller spectral radius~$>1$ than the reciprocal ones. For~$g=3$,
the reciprocal matrices realise a smaller spectral radius than the skew-reciprocal ones.
\end{thm}

\begin{thm}
\label{irred_srwins}
Let~$g\ge 1$. Among nonnegative irreducible~$A \in \mathrm{GL}_{2g}(\mathbb{Z})$, the 
skew-reciprocal ones realise a smaller spectral radius~$>1$ than the reciprocal ones. 
\end{thm}

Naturally, the following question arises by dropping the hypotheses of primitivity or irreducibility.

\begin{que}
\label{whowins_que}
Let~$g\ge1$. Do the skew-reciprocal matrices~$A \in \mathrm{GL}_{2g}(\mathbb{Z})$ realise a smaller 
spectral radius~$>1$ than the reciprocal ones?
\end{que}

Answering Question~\ref{whowins_que} positively for all~$g$ that are large enough powers of two 
would provide an independent proof of Dimitrov's theorem~\cite{Dimitrov}, also known as the conjecture 
of Schinzel and Zassenhaus~\cite{SZ}. 
This follows from an argument\footnote{The definitions for reciprocality and skew-reciprocality 
used in the article~\cite{L} is slightly different, prescribing the sign of the constant coefficient 
so that the polynomials arise from the action induced on the first homology by mapping classes. 
However, the argument presented there can be adopted directly to the definitions we  
use here.} previously given by the author~\cite{L}.
\smallskip

The proofs of Theorems~\ref{prim_srwins} and~\ref{irred_srwins} are based on McMullen's 
calculation of the minimal possible spectral radii~$>1$ for primitive and nonnegative irreducible
reciprocal matrices~\cite{McM}. In fact, we carry out the same calculation for skew-reciprocal 
matrices in order to determine the minimal spectral radii~$>1$ that arise, and compare 
the values with McMullen's result. The following two theorems summarise our results 
on these minimal spectral radii. 

\begin{thm}
\label{irred_minima}
Let~$g\ge1$. The minimal spectral radius~$>1$ 
among skew-reciprocal nonnegative irreducible matrices~$A\in\mathrm{GL}_{2g}(\mathbb{Z})$ 
is realised by the by the largest root~$\lambda_{2g}$ of 
the polynomial 
\[ t^{2g} - t^g -1 \]
in case~$g$ is odd, and of the polynomial
\[ t^{2g} -t^{g+1}-t^{g-1} -1 \]
in case~$g$ is even. 
\end{thm}

\begin{thm}
\label{prim_minima}
Let~$g\ge2$. 
The minimal spectral radius~$>1$
among skew-reciprocal primitive matrices~$A\in\mathrm{GL}_{2g}(\mathbb{Z})$
is realised by the by the largest root~$\mu_{2g}$ of 
the polynomial 
\[ t^{2g} - t^{g+2}-t^{g-2} -1 \]
in case~$g$ is odd, and of the polynomial
\[ t^{2g} -t^{g+1}-t^{g-1} -1 \]
in case~$g$ is even. 
\end{thm}

Given Theorems~\ref{prim_minima} and~\ref{irred_minima}, Theorems~\ref{prim_srwins} and~\ref{irred_srwins} follow readily.

\begin{proof}[Proof of Theorems~\ref{prim_srwins} and~\ref{irred_srwins}]
The normalised sequence~$(\mu_{2g})^{2g}$ converges to the square 
of the silver ratio,~$(1+\sqrt{2})^2 = 3+2\sqrt{2}$, while the normalised sequence~$(\lambda_{2g})^{2g}$
has two accumulation points: again the square of the silver ratio, for even~$g$, but 
also the square of the golden ratio,~$\left(\frac{1+\sqrt{5}}{2}\right)^2 = \frac{3+\sqrt{5}}{2}$, for odd~$g$. 
To see this, note that~$(\mu_{2g})^g$ is the largest real zero of the function
\[ f(t) = t^{2} - t^{1+\frac{2}{g}}-t^{1-\frac{2}{g}} -1 \]
in case~$g$ is odd, and of the function
\[ f(t) = t^{2} -t^{1+\frac{1}{g}}-t^{1-\frac{1}{g}} -1 \]
in case~$g$ is even. Clearly for~$g\to\infty$ any real zero~$> 1$ converges to the larger root of 
the polynomial~$t^2-2t-1$, which is~$1+\sqrt{2}$. Therefore~$(\mu_{2g})^{2g}$ converges 
to~$(1+\sqrt{2})^2 = 3+2\sqrt{2}$. 
The argument for the sequence~$\lambda_g$ is analogous. 

In either case, these 
accumulation points are all smaller than the analogous smallest possible accumulation 
point in the case of reciprocal matrices. Indeed, McMullen proves that the minimal 
accumulation point for the normalised sequence of spectral radii for reciprocal 
matrices is~$\varphi^4$, where~$\varphi$ is the golden ratio~\cite{McM}. This number is 
strictly larger than the square of the silver ratio, so asymptotically the result is given. 
We finish the proof Theorems~\ref{prim_srwins} and~\ref{irred_srwins} by using 
the monotonicity of the sequences of normalised spectral radii, and compare these 
sequences for small~$g$. It turns out that the only case where a normalised sequence 
of the skew-reciprocal matrices is larger than the accumulation point~$\varphi^4$ of the 
normalised sequence of the reciprocal matrices is in the case~$g=3$ of primitive matrices. 
This finishes the proof for~$g\ne 3$. For~$g=3$ in the primitive case, we simply check 
that~$(\mu_6)^6 > 8.18$, whereas the smallest normalised spectral radius in the 
reciprocal case is~$\approx 7.57$ by McMullen's result~\cite{McM}. This finishes the 
proof also in the case~$g=3$. 
\end{proof}

\emph{Applications to pseudo-Anosov stretch factors}. McMullen's result on the 
 minimal spectral radii for reciprocal matrices~\cite{McM} is interesting in the context of 
 minimal stretch factors of pseudo-Anosov mapping classes. For example, it is used by Hironaka 
 and Tsang to find, for a large class of examples, the optimal lower bound for normalised 
 pseudo-Anosov stretch factors in the fully-punctured case~\cite{HiTs}. Naturally, we  
 hope our Theorem~\ref{prim_minima} will be instrumental for an analogous result in the case of 
 orientation-reversing pseudo-Anosov mapping classes. In fact, Theorem~\ref{prim_minima} 
 also presents some of the same polynomials found by the author and Strenner~\cite{LS0} in the 
 search of minimal stretch factors among orientation-reversing pseudo-Anosov mapping 
 classes with orientable invariant foliations.
\smallskip

\emph{Organisation}. In the next section, we minimally review the notion of the clique polynomial 
as well as the input we need from McMullen's work~\cite{McM}, before proving 
Theorems~\ref{irred_minima} and~\ref{prim_minima} in the third and final section. 
\smallskip

\emph{Acknowledgments}. The author thanks Chi Cheuk Tsang and an anonymous 
referee for their comments on a first version of this article. 

\section{The clique polynomial}
In this section, we review parts of McMullen's technique using the curve graph and its 
clique polynomial in order to single out minimal spectral radii among nonnegative 
matrices. We try to keep the discussion as concise as possible and refer to the 
original article~\cite{McM} and the references therein for a more complete discussion.

Let~$\Gamma$ be a directed graph. A \emph{simple closed curve} in~$\Gamma$ is the 
union of directed edges describing a closed directed loop in~$\Gamma$ that visits 
every vertex at most once. The \emph{curve graph}~$G$ of~$\Gamma$ is obtained 
as follows: there is a vertex for every simple closed curve in~$\Gamma$, and two 
vertices are connected by an edge if and only if the corresponding simple closed curves 
have no vertex of~$\Gamma$ in common. Each vertex of~$G$ is given a weight 
describing the number of edges contained in the simple closed curve. 

A subset~$K$ of the vertices of~$G$ is a \emph{clique} if the subgraph induced by~$K$ 
is complete. The \emph{clique polynomial} of~$G$ is defined to be 
\[ 
Q(t) = \sum_K (-1)^{|K|} t^{w(K)},
\]
where we also allow~$K=\emptyset$, and~$w(K)$ is the sum of all weights 
of vertices in~$K$. 

With a nonnegative square matrix~$A$ of dimension~$n\times n$, we associate 
a directed graph~$\Gamma_A$ that has~$n$ vertices and directed edges between 
the vertices according to the coefficients of~$A$. Let~$G_A$ be the associated 
curve graph and let~$Q_A(t)$ be its clique polynomial. By a well-known result in 
graph theory, the characteristic polynomial of~$A$ is the reciprocal of~$Q_A(t)$, 
that is,~$\chi_A(t) = t^{n} Q_A(t^{-1})$. In particular, the spectral 
radius of~$A$ equals the inverse of the smallest modulus among the roots 
of~$Q_A(t)$. 

\subsection{McMullen's classification of graphs with small growth}

McMullen defines a minimal growth rate~$\lambda(G)$ for graphs~$G$. 
We refer to McMullen's original article~\cite{McM} for more details. The only statement we 
need for our purposes is that if~$A$ is a nonnegative matrix of dimension~$n\times n$ 
and with spectral radius~$\rho(A)$, then~$\lambda(G_A)$ is a lower bound 
for the normalised spectral radius~$\rho(A)^n$.

\begin{figure}[h]
\def\svgwidth{255pt}
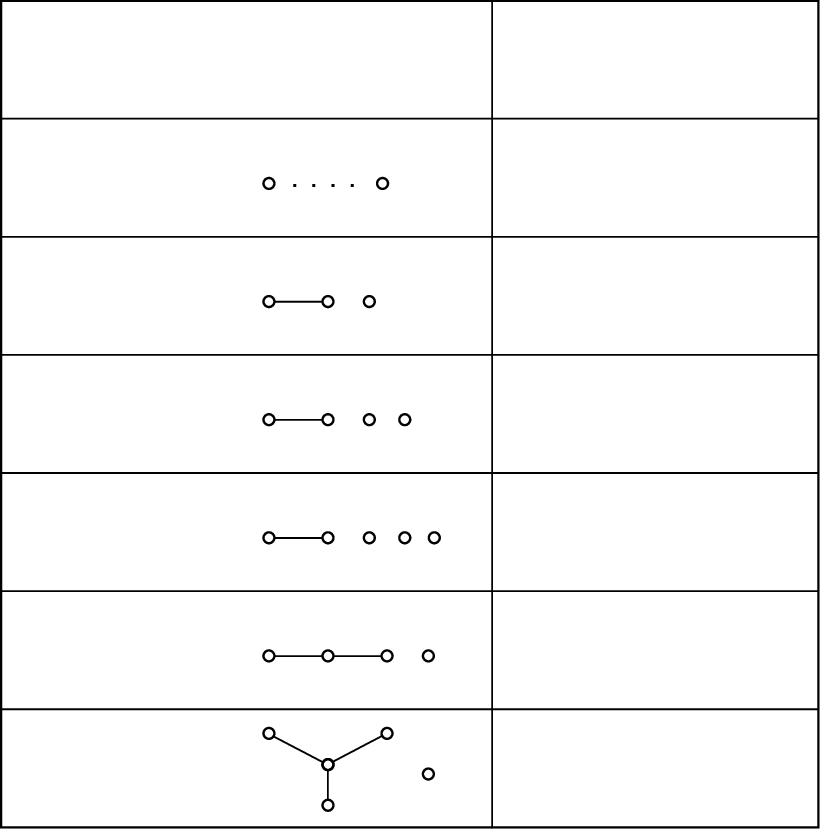
\caption{Some graphs and their minimal growth rates.}
\label{minimalrates}
\end{figure}

We are interested in irreducible matrices~$A$. This means that the 
directed graph~$\Gamma_A$ is strongly connected, which in turn implies that the 
associated curve graph~$G_A$ has complement~$G_A'$ that is connected. 
Here, the complement~$G'_A$ is defined to be the graph with the same vertex 
set as~$G_A$ but the complementary edges. There are very few curve graphs~$G$ 
with~$G'$ connected and minimal growth rate~$\lambda(G)<8$. Our argument is based 
on the following classification due to McMullen~\cite{McM}. 

\begin{thm}[Theorem 1.6 in McMullen~\cite{McM}]
\label{mcm_classif}
The graphs~$G$ with~$G'$ connected and~$1<\lambda(G)<8$ are given by 
\[
A_2^\ast, A_2^{\ast\ast}, A_2^{\ast\ast\ast}, A_3^\ast, Y^\ast \hspace{0.2cm} \mathrm{and} 
\hspace{0.2cm} nA_1,
\]
For~$2\le n\le 7.$
\end{thm}

This result tells us that if we want to describe all irreducible nonnegative matrices with 
normalised spectral radius~$<8$, all we have to do is check among those whose 
associated curve graph is among the ones shown in Figure~\ref{minimalrates}. 
In fact, in Section~\ref{proof_sec} we split the proof of Theorems~\ref{irred_minima} 
and~\ref{prim_minima} into four propositions, dealing with the four distinct cases. 
In each case, the first thing we do is to realise the proposed minimal normalised spectral radius, which  
(except in one case for~$g=3$) turns out to be~$<8$. Then McMullen's classification 
result applies and we only have to check the curve graphs given in Theorem~\ref{mcm_classif} to finish the proof. 

\section{Skew-reciprocity and minimal spectral radii}
\label{proof_sec}
In this section, we prove Theorems~\ref{irred_minima} and~\ref{prim_minima}. We 
break down the proof into four separate propositions, distinguishing between the irreducible 
and the primitive case, as well as the case of even and odd~$g$. 

The condition of skew-reciprocity poses slightly different constraints on the 
coefficients of the polynomial than reciprocity. First of all,
we note that if the roots of a polynomial~$f \in \mathbb{Z}[t]$ are invariant under 
the transformation~$\lambda\mapsto -\lambda^{-1}$, then we  
have~$f(t) = \pm t^{\mathrm{deg}(f)}f(-t^{-1})$. This entails the following constraints. 

\begin{lem}
\label{srec_constraints}
Let~$f\in\mathbb{Z}[t]$ be a monic skew-reciprocal polynomial of degree~$2g$. 
Then we have the following conditions on the coefficients of~$f$: 
\begin{enumerate}
\item the moduli of the coefficients of~$t^d$ and~$t^{2g-d}$ agree. More precisely,
\item if~$g$ is even and~$f(0) = 1$, the coefficients of~$t^d$ and~$t^{2g-d}$ 
agree for even~$d$ and differ by a sign for odd~$d$,
\item if~$g$ is even and~$f(0) = -1$, the coefficients of~$t^d$ and~$t^{2g-d}$ 
agree for odd~$d$ and differ by a sign for even~$d$. In particular, the middle 
coefficient of~$f$ vanishes,
\item if~$g$ is odd and~$f(0) = 1$, the coefficients of~$t^d$ and~$t^{2g-d}$ 
agree for even~$d$ and differ by a sign for odd~$d$. In particular, the middle 
coefficient of~$f$ vanishes,
\item if~$g$ is odd and~$f(0) = -1$, the coefficients of~$t^d$ and~$t^{2g-d}$ 
agree for odd~$d$ and differ by a sign for even~$d$.
\end{enumerate}
\end{lem}

\begin{proof}
Let~$f(t) = a_{2g}t^{2g} + \cdots + a_0$ be a skew-reciprocal polynomial. 
The polynomial relation~$f(t) = \pm t^{2g}f(-t^{-1})$ given by skew-reciprocity
translates to the relation~$$a_d = \pm (-1)^{2g-d} a_{2g-d} = \pm (-1)^da_{2g-d}$$ 
for each pair coefficients~$a_d$ and~$a_{2g-d}$. Clearly, the coefficients are 
symmetric up to a possible sign that alternates between~$+1$ and~$-1$ as we change
the index~$d$ of the coefficient by one. In particular, the sign is the same 
for all even~$d$ and it is the same for all odd~$d$. Now recall that~$f(t)$ is monic, 
that is,~$a_{2g}=1$. In this case,~$f(0)=1$ means that 
the coefficients~$a_d$ and~$a_{2g-d}$ agree for~$d=0$ and hence all even~$d$, 
and they differ by a sign for odd~$d$. Similarly,~$f(0)=-1$ means that 
the coefficients differ by a sign for~$d=0$ and hence for all even~$d$, and they 
agree for even~$d$. Finally, the middle coefficient~$a_g$ needs to be zero if~$a_d$ and~$a_{2g-d}$ 
differ by a sign for all~$d$ with the same parity as~$g$. This distinction yields 
the four different cases (2)--(5) described in the statement of Lemma~\ref{srec_constraints}. 
\end{proof}

\begin{example}
\label{example}
\emph{
To see the main proof ideas applied to the simplest nontrivial example, we now determine which spectral radii are 
obtained by skew-reciprocal matrices~$A\in\mathrm{GL}_{2g}(\mathbb{Z})$ with curve graph~$G_A = 2A_1$, 
which is the graph with two isolated vertices. In this case~$Q(t) = 1-t^a-t^b$, where~$a$ and~$b$ 
are the weights of the vertices. The polynomial~$Q(t)$ is of degree~$2g$, and without loss of 
generality we assume~$b=2g$. The only way to have the moduli of the coefficients symmetrically 
distributed as in (1) of Lemma~\ref{srec_constraints} is if~$a=g$. Therefore, the only clique polynomial 
we possibly obtain is~$Q(t) = 1-t^g-t^{2g}$. Hence, the only characteristic polynomial we possibly 
obtain is~$t^{2g} - t^g -1$. We make the following observations: 
\begin{enumerate}
\item if~$A$ is primitive, then~$g=1$. Indeed, otherwise the characteristic polynomial is a polynomial in~$t^g$ 
with~$g>1$. Such a polynomial cannot be the characteristic polynomial of a primitive matrix. On the other 
hand, for~$g=1$, the polynomial~$t^2-t-1$ is the minimal polynomial of the golden ratio, realised as the 
characteristic polynomial of the matrix~$\begin{pmatrix} 0 & 1 \\ 1 & 1 \end{pmatrix}$.
\item For~$g$ is odd, the polynomial~$t^{2g}-t^{g}-1$ is the characteristic polynomial of a 
nonnegative irreducible matrix in~$\mathrm{GL}_{2g}(\mathbb{Z})$, namely a standard companion matrix. 
For example,~$t^6-t^3-1$ is the characteristic polynomial of the matrix
\[
\begin{pmatrix}
0 & 1 & 0 & 0 & 0 & 0 \\
0 & 0 & 1 & 0 & 0 & 0 \\
0 & 0 & 0 & 1 & 0 & 0 \\
0 & 0 & 0 & 0 & 1 & 0 \\
0 & 0 & 0 & 0 & 0 & 1 \\
1 & 0 & 0 & 1 & 0 & 0
\end{pmatrix},
\]
which is irreducible.
\item For even~$g$, we note that since the constant coefficient of the characteristic polynomial of~$A$ 
is negative, we are in case (3) of Lemma~\ref{srec_constraints}. In particular, the coefficient of~$t^g$ should be zero 
instead of~$-1$. This means that for even~$g$, there are no skew-reciprocal matrices~$A\in\mathrm{GL}_{2g}(\mathbb{Z})$ 
with curve graph~$2A_1$.
\end{enumerate}
In summary, we obtain that the following spectral radii can be realised for 
matrices~$A\in\mathrm{GL}_{2g}(\mathbb{Z})$ with curve graph~$2A_1$: 
\begin{enumerate}
\item[(i)] Among primitive skew-reciprocal matrices~$A\in\mathrm{GL}_{2g}(\mathbb{Z})$, 
only the golden ratio is realised, for~$g=1$. For~$g>1$, there are no primitive skew-reciprocal matrices with 
curve graph~$2A_1$. 
\item[(ii)] Among nonnegative irreducible skew-reciprocal~$A\in\mathrm{GL}_{2g}(\mathbb{Z})$, 
the largest root of the polynomial~$t^{2g}-t^g-1$ is realised, for odd~$g$. For even~$g$, there are no
nonnegative irreducible skew-reciprocal matrices~$A\in\mathrm{GL}_{2g}(\mathbb{Z})$ with curve graph~$2A_1$.  
\end{enumerate}
}
\end{example}

\subsection{The irreducible case}

\subsubsection{The case of odd~$g$}

\begin{prop}
Let~$g\ge1$ odd.
Among all skew-reciprocal nonegative irreducible matrices~$A\in\mathrm{GL}_{2g}(\mathbb{Z})$, 
the minimal spectral radius~$>1$ is realised by the largest root~$\lambda_{2g}$ 
of the polynomial~$t^{2g}-t^{g}-1$. 
\end{prop}

\begin{proof}
Let~$A\in\mathrm{GL}_{2g}(\mathbb{Z})$ be a nonnegative skew-reciprocal matrix. 
Then its square~$A^2$ is a reciprocal matrix. In particular, by McMullen's result on 
minimal normalised spectral radii for nonnegative reciprocal matrices~\cite{McM} we know that 
its normalised spectral radius must be at least~$\varphi^4$. Therefore, the normalised 
spectral radius of~$A$ must be at least~$\varphi^2$, which incidentally 
equals~$(\lambda_{2g})^{2g}$. In order to finish the proof, it therefore suffices to realise 
the polynomial~$t^{2g}-t^{g}-1$ as the characteristic polynomial of a nonnegative irreducible matrix 
in~$\mathrm{GL}_{2g}(\mathbb{Z})$. 
This is straightforward, as it can be achieved by a standard companion matrix, 
see (2) in Example~\ref{example} for the case~$g=3$. 
\end{proof}

\begin{remark}\emph{
It actually follows from McMullen's classification that the polynomial~$t^{2g}-t^{g}-1$ 
is the unique characteristic polynomial that can appear for a matrix 
that minimises the spectral radius. Indeed, only the graph~$2A_1$ can appear as curve graph, 
with clique polynomial~$1-t^a-t^b$. For this polynomial to be skew-reciprocal we must
either have~$a=2g$ and~$b=g$ or~$b=2g$ and~$a=g$. 
Both cases yield our candidate polynomial.}
\end{remark}

\subsubsection{The case of even~$g$}

We note that the above proof does not work for even~$g$: if~$g$ is even, then 
the polynomial~$t^{2g}-t^{g}-1$ is not skew-reciprocal, as noted in Example~\ref{example}.
We instead have the following minimisers.

\begin{prop}
\label{irred_even_prop}
Let~$g\ge2$ even.
Among all skew-reciprocal nonnegative irreducible matrices~$A\in\mathrm{GL}_{2g}(\mathbb{Z})$, 
the minimal spectral radius~$>1$ is realised by the largest root~$\lambda_{2g}$ 
of the polynomial~$t^{2g}-t^{g+1}-t^{g-1}-1$. 
\end{prop}

\begin{proof}
The largest root~$\lambda_{2g}$ of the polynomial~$t^{2g}-t^{g+1}-t^{g-1}-1$ is clearly 
realised as the spectral radius of a nonnegative irreducible~$A\in\mathrm{GL}_{2g}(\mathbb{Z})$. 
Indeed, again we can achieve this by a standard companion matrix. We now note 
that~$(\lambda_{2g})^{2g}$ is a descending sequence converging to~$3+2\sqrt{2}$ and 
starting at~$\varphi^4 < 7$ for~$g=2$. In particular, we can finish the proof by showing 
that for even~$g\ge2$,~$\lambda_{2g}$ minimises the spectral radius among all skew-reciprocal 
nonnegative irreducible matrices~$A\in\mathrm{GL}_{2g}(\mathbb{Z})$ that have one of the 
graphs in Figure~\ref{minimalrates} except~$A_2^{\ast\ast\ast}$ or~$Y^\ast$ as their curve graph.

\begin{enumerate}
\item $G=2A_1$. As noted in Example~\ref{example}, there are no nonnegative irreducible skew-reciprocal
matrices~$A\in\mathrm{GL}_{2g}(\mathbb{Z})$ with~$g$ even and~$2A_1$ as their curve graph.
\item $G=3A_1$. In this case~$Q(t) = 1-t^a-t^b-t^c$. Without loss of generality 
we assume~$c=2g$. If we want the moduli of the coefficients symmetrically 
distributed, we are left with the options
\[
1-t^{g-d}-t^{g+d}-t^{2g}
\] 
for~$0\le d\le g$. The case~$d=g$ is ruled out as the resulting polynomial is a monomial. 
The case~$d=0$ is ruled out by Lemma~\ref{srec_constraints}, as above. 
It follows that~$1<d<g$. By Proposition~$3.2$ in~\cite{LS}, we know that the largest root 
of the reciprocal polynomial~$t^{2g}-t^{g+d}-t^{g-d}-1$ is a strictly increasing function of~$d$. 
So, the smallest spectral is obtained for~$d=1$, resulting in our candidate 
polynomial~$t^{2g}-t^{g+1}-t^{g-1}-1$. 
\item $G=4A_1$ or~$G=6A_1$. As for~$G=2A_1$, the number of terms is odd. 
The only way to have the moduli of the coefficients symmetrically 
distributed is to have a middle coefficient, a contradiction to Lemma~\ref{srec_constraints}. 
\item $G=5A_1$. In this case~$Q(t) = 1-t^a-t^b-t^c-t^d-t^e$. Without loss of generality~$e=2g$. 
As in the case~$G=3A_1$, there must be at least two paired terms of power~$\ne 0,g$. 
We assume without loss of generality that~$0<a<g<b = 2g-a<2g$. Since the polynomial reciprocal 
to~$Q(t)$ is realised by a standard companion matrix, we can delete its coefficients 
that correspond to the terms~$-t^c$ and~$-t^d$ and obtain a matrix with strictly smaller 
spectral radius and characteristic polynomial~$t^{2g}-t^b-t^a-1$, 
where the~$a,b$ and~$g$ satisfy~$0<a<g<b = 2g-a<2g.$ We have shown in the case~$G=3A_1$ that the minimal 
spectral radius obtained by a matrix with such a characteristic polynomial is our 
candidate~$\lambda_{2g}$. 
\item $G=A_2^\ast$. In this case,~$Q(t) = 1-t^a - t^b - t^c + t^{a+b}$, where~$c$ is the weight 
on the isolated vertex of~$A_2^\ast$. Since there are three terms not paired with the constant
term, we deduce there must be a nonvanishing middle coefficient. 
By Lemma~\ref{srec_constraints}, this means 
that the leading coefficient of~$Q(t)$ is positive, so we have~$a+b = 2g$ and~$c=g$. 
We note that the possibilities that remain are~$1-t^{g-d}-t^g -t^{g+d}+ t^{2g}$, which are 
reciprocal. Theorem 7.3 by McMullen~\cite{McM} provides that the normalised spectral 
radius is~$\ge \varphi^4 = \lambda_4$ and~$>\lambda_{2g}$ for~$g>2$. 
\item $G=A_2^{\ast\ast}$. In this case,~$Q(t) = 1- t^a-t^b - t^c - t^d + t^{a+b}$, 
where~$c$ and~$d$ are the weights on the isolated vertices of~$A_2^{\ast\ast}$. 
If the leading coefficient of~$Q(t)$ is positive, then~$Q(t)$ is reciprocal. By McMullen's 
analysis of the curve graph~$A_2^{\ast\ast}$ for reciprocal weights, a normalised spectral 
radius in this case must be~$\ge (2+\sqrt{3})^2 > \lambda_{2g}$. It remains to consider the case 
where the leading coefficient of~$Q(t)$ is negative. Without loss of generality, 
we assume~$c=2g$. We have the following conditions on the other parameters~$a,b,d$.
\begin{itemize}
\item $a+b < 2g$ and hence~$a+b \le 2g-2$. Indeed,~$t^{a+b}$ appears with a positive sign
and must be paired with a term~$-t^x$ with negative sign, for~$x=a,b$ or~$d$. In particular,
Lemma~\ref{srec_constraints} implies that~$a+b$ must be even.
\item either~$d+a = 2g$ or~$d+b = 2g$. We assume without loss of generality that~$d+b = 2g$. 
Thus,~$1\le a < b,d < a+b \le 2g-2$. 
\end{itemize}
Now consider the directed graph~$\Gamma_{a,b,d}$,  
\begin{figure}[h]
\def\svgwidth{130pt}
%% Creator: Inkscape 1.1.1 (c3084ef, 2021-09-22), www.inkscape.org
%% Accompanies image file 'Gamma_abd.eps' (pdf, eps, ps)
\begingroup%
  \makeatletter%
  \providecommand\color[2][]{%
    \errmessage{(Inkscape) Color is used for the text in Inkscape, but the package 'color.sty' is not loaded}%
    \renewcommand\color[2][]{}%
  }%
  \providecommand\transparent[1]{%
    \errmessage{(Inkscape) Transparency is used (non-zero) for the text in Inkscape, but the package 'transparent.sty' is not loaded}%
    \renewcommand\transparent[1]{}%
  }%
  \providecommand\rotatebox[2]{#2}%
  \newcommand*\fsize{\dimexpr\f@size pt\relax}%
  \newcommand*\lineheight[1]{\fontsize{\fsize}{#1\fsize}\selectfont}%
  \ifx\svgwidth\undefined%
    \setlength{\unitlength}{210.19276821bp}%
    \ifx\svgscale\undefined%
      \relax%
    \else%
      \setlength{\unitlength}{\unitlength * \real{\svgscale}}%
    \fi%
  \else%
    \setlength{\unitlength}{\svgwidth}%
  \fi%
  \global\let\svgwidth\undefined%
  \global\let\svgscale\undefined%
  \makeatother%
  \begin{picture}(1,0.98887066)%
    \lineheight{1}%
    \setlength\tabcolsep{0pt}%
    \put(0,0){\includegraphics[width=\unitlength]{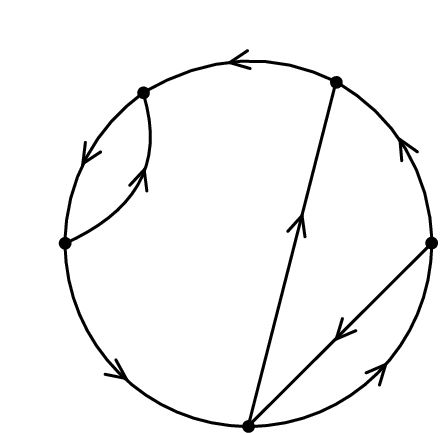}}%
    \put(0.92180472,0.08203801){\makebox(0,0)[lt]{\lineheight{1.25}\smash{\begin{tabular}[t]{l}$b-1$\end{tabular}}}}%
    \put(0.96167021,0.69197235){\makebox(0,0)[lt]{\lineheight{1.25}\smash{\begin{tabular}[t]{l}2\end{tabular}}}}%
    \put(0.30741587,0.8959634){\makebox(0,0)[lt]{\lineheight{1.25}\smash{\begin{tabular}[t]{l}$2g-a-b-1$\end{tabular}}}}%
    \put(0.02160946,0.68619438){\makebox(0,0)[lt]{\lineheight{1.25}\smash{\begin{tabular}[t]{l}$a-1$\end{tabular}}}}%
  \end{picture}%
\endgroup%

\caption{The directed graph~$\Gamma_{a,b,d}$.}
\label{Gamma_abd}
\end{figure}
where a weight~$w$ on an edge indicates~$w-1$ additional vertices placed on the edge. 
We note that the clique polynomial of the curve graph of ~$\Gamma_{a,b,d}$ is~$Q(t)$. 
Furthermore, deleting 
the edge of length~$1$ in~$\Gamma_{a,b,d}$ that forms the simple closed curve of 
length~$a$ strictly decreases 
the spectral radius of the associated adjacency matrix. Furthermore, the new curve 
graph is~$3A_1$ and the new clique polynomial is obtained by removing the terms~$-t^a$ 
and~$t^{a+b}$, and hence skew-reciprocal. This is a case we have already dealt with.
\item $G=A_3^\ast$. In this case,~$Q(t) = 1- t^a-t^b - t^c - t^d + t^{a+b}+t^{b+c}$, where~$d$
is the weigth of the isolated vertex and~$b$ is the weight of the vertex of degree 
two of~$A_3^\ast$. Since the number of summands is odd, the middle term must have a 
nonvanishing coefficient. By Lemma~\ref{srec_constraints}, the leading coefficient of~$Q(t)$ is 
positive, and we assume without loss of generality that~$b+c = 2g$. Note that since the 
coefficients of~$t^b$ and~$t^c$ have the same sign,~$b$ and~$c$ need to be even. 
We now distinguish three cases: either~$a+b = g$,~$a = g$, or~$d=g$. 
\begin{itemize}
\item if~$a+b = g$, then~$Q(t) = 1 - t^a - t^b + t^g - t^{2g-a} -t^{2g-b} + t^{2g}$, which is 
reciprocal. By McMullen's analysis of the curve graph~$A_3^\ast$ for reciprocal weights~\cite{McM}, 
the normalised spectral radius must either be~$>12.5 > \varphi^4$, or~$Q(t)$ is among the examples 
arising from~$A_2^\ast$, a case we have dealt with already.
\item if~$a=g$, then~$d+a+b = 2g$.  Since~$b$ and~$g$ are even, so must be~$a+b = g+b$, 
and hence also~$d$. By Lemma~\ref{srec_constraints}, the coefficients of~$t^d$ and 
of~$t^{a+b}$ should have the same sign, a contradiction. 
\item if~$d=g$, we have~$2a + b = 2g$, and hence~$a\le g-1$. 
Let~$\Gamma'_{a,b,c}$be defined as in Figure~\ref{Gamma2_abc},
\begin{figure}[h]
\def\svgwidth{240pt}
%% Creator: Inkscape 1.1.1 (c3084ef, 2021-09-22), www.inkscape.org
%% Accompanies image file 'Gamma2_abc.eps' (pdf, eps, ps)
\begingroup%
  \makeatletter%
  \providecommand\color[2][]{%
    \errmessage{(Inkscape) Color is used for the text in Inkscape, but the package 'color.sty' is not loaded}%
    \renewcommand\color[2][]{}%
  }%
  \providecommand\transparent[1]{%
    \errmessage{(Inkscape) Transparency is used (non-zero) for the text in Inkscape, but the package 'transparent.sty' is not loaded}%
    \renewcommand\transparent[1]{}%
  }%
  \providecommand\rotatebox[2]{#2}%
  \newcommand*\fsize{\dimexpr\f@size pt\relax}%
  \newcommand*\lineheight[1]{\fontsize{\fsize}{#1\fsize}\selectfont}%
  \ifx\svgwidth\undefined%
    \setlength{\unitlength}{302.43145046bp}%
    \ifx\svgscale\undefined%
      \relax%
    \else%
      \setlength{\unitlength}{\unitlength * \real{\svgscale}}%
    \fi%
  \else%
    \setlength{\unitlength}{\svgwidth}%
  \fi%
  \global\let\svgwidth\undefined%
  \global\let\svgscale\undefined%
  \makeatother%
  \begin{picture}(1,0.55257245)%
    \lineheight{1}%
    \setlength\tabcolsep{0pt}%
    \put(0,0){\includegraphics[width=\unitlength]{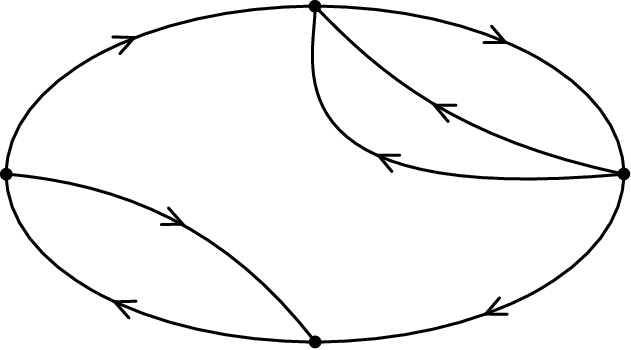}}%
    \put(0.00818432,0.02206673){\makebox(0,0)[lt]{\lineheight{1.25}\smash{\begin{tabular}[t]{l}$g-a-1$\end{tabular}}}}%
    \put(0.19665684,0.26509711){\makebox(0,0)[lt]{\lineheight{1.25}\smash{\begin{tabular}[t]{l}$b+a+1-g$\end{tabular}}}}%
    \put(0.83731881,0.4913928){\makebox(0,0)[lt]{\lineheight{1.25}\smash{\begin{tabular}[t]{l}$a-1$\end{tabular}}}}%
    \put(0.65157812,0.22107392){\makebox(0,0)[lt]{\lineheight{1.25}\smash{\begin{tabular}[t]{l}$c-a+1$\end{tabular}}}}%
  \end{picture}%
\endgroup%

\caption{The directed graph~$\Gamma'_{a,b,c}$.}
\label{Gamma2_abc}
\end{figure}
where the edge of weight~$g-a-1$ is contracted in case~$a=g-1$. We note that the 
clique polynomial of the curve graph of~$\Gamma'_{a,b,c}$ is exactly our~$Q(t)$, 
where~$d=g$. Deleting the edge of length one that forms the simple closed curve of length~$a$
strictly decreases the spectral radius of the associated adjacency matrix. The clique polynomial 
we obtain after this deletion of an edge is of the form~$1-t^b-t^g -t^{2g-b} + t^{2g}$,
a case we have dealt with in our study of~$G=A_2^\ast$. 
\end{itemize}
\end{enumerate}
\end{proof}

\begin{remark}
\label{uniqueness_even}
\emph{
Our proof of Proposition~\ref{irred_even_prop} actually shows that for~$g\ne2$, 
the polynomial $t^{2g}-t^{g+1}-t^{g-1}-1$
is the unique characteristic polynomial that can appear for a matrix minimising the spectral 
radius. Except for~$g=2$, where we have a second possibility (appearing in (5)) for the characteristic 
polynomial:~$t^4-3t^2+1$. In the case~$g=2$, both minimising polynomials are divisible by the minimal polynomial of 
the golden ratio.}
\end{remark}

\begin{remark}
\emph{
In the cases (4), (6) and (7) of the proof of Proposition~\ref{irred_even_prop}, 
we construct irreducible matrices and reduce some of their coefficients 
in order to obtain irreducible matrices with strictly smaller spectral radii that belong to other cases we have 
already dealt with. For a quicker proof of Proposition~\ref{irred_even_prop}, we could use the
monotonicity property for the spectral radius formulated by McMullen~\cite{McM} on the level of 
the weighted curve graph. However, this monotonicity is not strict in general. In particular, this 
proof strategy seems to fail to provide the uniqueness of the minimising characteristic polynomials 
described in Remark~\ref{uniqueness_even}.}
\end{remark}

\subsection{The primitive case}
\subsubsection{The case of even~$g$} 

\begin{prop}
Let~$g\ge2$ even.
Among all skew-reciprocal primitive matrices~$A\in\mathrm{GL}_{2g}(\mathbb{Z})$, 
the minimal spectral radius~$>1$ is realised by the largest root~$\mu_{2g}$ 
of the polynomial~$t^{2g}-t^{g+1}-t^{g-1}-1$. 
\end{prop}

\begin{proof}
By Proposition~\ref{irred_even_prop}, we know that~$\mu_{2g}$ is actually the minimal 
spectral radius among all nonnegative irreducible matrices. In order to prove the result, it is enough to 
show that the polynomial~$t^{2g}-t^{g+1}-t^{g-1}-1$ is the characteristic polynomial of a 
primitive matrix. This is the case. Indeed, we can take the standard companion matrix for 
the polynomial and draw its directed adjacency graph. We directly see that there are 
directed cycles of length~$g-1,g+1$ and~$2g$. In order to show that the matrix is primitive, 
it suffices to show that their common greatest divisor is~$1$. Let~$n$ be a positive integer 
that divides both~$2g$ and~$g+1$. Since~$g$ is even,~$g+1$ is odd and so~$n$ has to 
to be odd as well. Now since~$n$ is odd and divides~$2g$, it divides~$g$. We have 
that~$n$ divides both~$g$ and~$g+1$ and therefore~$n=1$. This finishes the proof.
\end{proof}

\subsubsection{The case of odd~$g$}

\begin{prop}
Let~$g\ge3$ odd.
Among all skew-reciprocal primitive matrices~$A\in\mathrm{GL}_{2g}(\mathbb{Z})$, 
the minimal spectral radius~$>1$ is realised by the largest root~$\mu_{2g}$ 
of the polynomial~$t^{2g}-t^{g+2}-t^{g-2}-1$. 
\end{prop}

\begin{proof}[Proof for~$g\ge5$.]
We take a standard companion matrix to realise the largest root~$\mu_{2g}$ of the 
polynomial~$t^{2g}-t^{g+2}-t^{g-2}-1$ as a spectral radius of a nonnegative matrix 
in~$\mathrm{GL}_{2g}(\mathbb{Z})$. Furthermore, 
the associated directed graph has simple closed curves of lengths~$2g,g+2,g-2$, which have 
greatest common divisor~$1$. Indeed, since~$g$ is odd so is~$g+2$, so if~$n$ divides both~$2g$ 
and~$g+2$, then it must be odd itself and hence divide~$g$. But then, since~$n$ divides 
both~$g$ and~$g+2$, it must divide~$2$. But~$n$ being odd implies~$n=1$. This shows that the 
companion matrix we constructed is primitive.

We now note 
that~$(\mu_{2g})^{2g}$ is a descending sequence converging to~$3+2\sqrt{2}$, 
with first values~$\mu_6 \approx 8.19$ and~$\mu_{10} \approx 6.42$. The example 
in the case~$g=3$ is too large to be covered by McMullen's classification, and we give a separate 
argument covering this case below. For~$g\ge5$ odd, we can proceed as before, and finish the proof by 
showing that~$\mu_{2g}$ minimises the spectral radius among all skew-reciprocal primitive 
matrices~$A\in\mathrm{GL}_{2g}(\mathbb{Z})$ with one of the graphs in Figure~\ref{minimalrates}
except~$A_2^{\ast \ast \ast}$ or~$Y^\ast$ as their curve graph. 

\begin{enumerate}
\item $G=2A_1$. As we noted in Example~\ref{example}, there exist no primitive skew-reciprocal
matrices~$A\in\mathrm{GL}_{2g}(\mathbb{Z})$ with~$g>1$ and~$2A_1$ as their curve graph.

\item $G=3A_1$. In this case,~$Q(t) = 1-t^a-t^b-t^c$. Without loss of generality, we assume
that~$c=2g$, which implies~$a=2g-b$ if we want symmetrically distributed coefficients. We have 
multiple possibilities for~$a$: 
\begin{itemize}
\item $a=g$. In this case,~$Q(t) = 1-2t^{g}-t^{2g}$, which is not primitive.
\item $a = g-1$. In this case~$Q(t) = 1-t^{g-1}-t^{g+1}-t^{2g}$. Lemma~\ref{srec_constraints} implies 
that this polynomial is not skew-reciprocal. Indeed, for it to be skew-reciprocal, the coefficients 
of~$t^{g+1}$ and~$t^{g-1}$ would have to differ by a sign since~$g-1$ is even. 
\item $a=g-2$. This case gives exactly our candidate polynomial with largest root~$\mu_{2g}$. 
\item $a<g-2$. By Proposition~$3.2$ in~\cite{LS}, we know that the largest root 
of the reciprocal polynomial~$t^{2g}-t^{g+d}-t^{g-d}-1$ is a strictly increasing function of~$d$.
In particular, the spectral radii we obtain for~$a<g-2$ are strictly larger than our candidate. 
\end{itemize}
\item $G=4A_1$. In this case,~$Q(t) = 1-t^a-t^b-t^c-t^d.$ We realise the reciprocal of~$Q(t)$ as 
the characteristic polynomial of a standard companion matrix.  Since there are five terms, 
there must be a middle coefficient. Deleting this middle coefficient amounts to decreasing a 
coefficient of the companion matrix from~$1$ to~$0$, strictly reducing the spectral radius. 
After this modification, the polynomial is among the examples we have already dealt with in 
the case~$G=3A_1$. 
\item $G=5A_1$ or~$G=6A_1$. This case can be dealt with in the same way as the 
case~$G=4A_1$. We delete the coefficients of a pair of terms whose powers add to~$2g$ (in the 
case of~$G=5A_1$) and additionally the middle coefficient (in the case of~$G=6A_1$). 
\item $G=A_2^\ast$. In this case,~$Q(t) = 1 -t^a-t^b-t^c + t^{a+b}$, where~$c$ is the weight 
on the isolated vertex of~$A_2^\ast$. Since there are five terms, there must be a middle coefficient,
which by Lemma~\ref{srec_constraints} implies that the leading coefficient is negative. 
We therefore must have~$c=2g$ and we can assume without loss of generality that~$b=g$ to get 
a polynomial of the form~$$Q(t) = 1-t^{a} - t^g + t^{a+g} - t^{2g},$$ and in particular~$2a+g = 2g$. 
But this implies that~$g=2a$ is even, a contradiction.
\item $G=A_2^{\ast\ast}$.  In this case,~$Q(t) = 1- t^a-t^b - t^c - t^d + t^{a+b}$,
where~$c$ and~$d$ are the weights on the isolated vertices of~$A_2^{\ast\ast}$. 
If the leading coefficient of~$Q(t)$ is positive, then~$a+b = 2g$ and the resulting polynomial is 
reciprocal. By McMullen's analysis of the curve graph~$A_2^{\ast\ast}$ for reciprocal weights~\cite{McM}, 
a normalised spectral radius in this case must be~$\ge (2+\sqrt{3})^2 > \mu_{2g}$. It remains to 
consider the case of a negative leading coefficient. Without loss of generality, we assume~$c=2g$. 
In order to have symmetrically distributet moduli of the coefficients, we must either have~$2a+b=2g$ 
and~$b+d = 2g$ or~$2b+a=2g$ and~$a+d=2g$. Both cases imply that~$d$ is even, and hence so 
must be~$b$ (in the former case) or~$a$ (in the latter case). In both cases, we get a contradiction to 
Lemma~\ref{srec_constraints}, which states that the coefficients must differ by a sign for even 
powers. 
\item $G=A_3^\ast$. In this case,~$Q(t) = 1-t^a-t^b-t^c-t^d+t^{a+b}+t^{b+c}$. Since there are seven 
terms, there must be a nonvanishing middle coefficient. By Lemma~\ref{srec_constraints}, this can 
only happen if the leading coefficient is negative. We must have~$d=2g$ 
and get $$Q(t) = 1-t^a-t^b-t^c + t^{a+b} +t^{b+c} -t^{2g}.$$
We distinguish cases depending on which term has power~$g$. 
\begin{itemize}
\item if one among~$a,b$ and~$c$ equals~$g$, we have~$a,b,c \le g$ 
and furthermore~$a+b,b+c > g$. This implies that~$a+b, b+c$ and two among~$a,b,c$ are even by 
Lemma~\ref{srec_constraints}. But then clearly all among~$a,b,c$ are even, and hence 
is~$g$, a contradiction.
\item if~$a+b=g$, then~$a,b < g$ and hence~$c,b+c>g$. Also~$b+c<2g$ so we must 
have~$a+c= 2g = 2b+c$. In particular,~$a=2b$ is even, and hence so must be~$c$. This 
contradicts Lemma~\ref{srec_constraints}, which states that coefficients must differ by a sign 
for terms with even powers.  The argument for the case~$b+c=g$ is obtained by switching~$a$ and~$c$. 
\end{itemize} 
\end{enumerate}
\end{proof}

\begin{proof}[Proof for~$g=3$.]
The candidate polynomial~$t^6-t^5-t-1$ has maximal real root~$\mu_6 \approx 1.4196 > \sqrt{2}$, 
so we only need to check other polynomials with roots bounded from above by this number, and 
bounded from below by~$\sqrt{2}$. Indeed, our proof in the case~$g\ge 5$ shows that there is 
no spectral radius~$< 8^\frac{1}{6} = \sqrt{2}$ among skew-reciprocal primitive 
matrices~$A\in\mathrm{GL}_6(\mathbb{Z})$. We now distinguish cases depending on the 
determinant of such a matrix~$A$.

\emph{Case 1: $\det(A) = 1$.} In this case, the characteristic polynomial must have a factor~$(t^2+1)$. 
The reason for this is that the eigenvalues of a skew-reciprocal matrix come in groups:
\begin{itemize}
\item
if~$\lambda \not\in \mathbb{R}$, $\lambda \not= \pm i$ is an eigenvalue, 
then so are~$-\lambda^{-1}, \bar\lambda$ and~$-\bar\lambda^{-1}$. 
These four roots of the characteristic polynomial contribute~$+1$ to the determinant, 
\item
if~$\lambda \in \mathbb{R}$, $\lambda\ne 0$ is an eigenvalue, then so is~$-\lambda^{-1}$. 
These two roots of the characteristic polynomial contribute~$-1$ to the determinant, 
\item 
if~$\lambda = \pm i$ is an eigenvalue, then so is~$\bar\lambda = -\lambda$. These two roots 
contribute~$+1$ to the determinant.
\end{itemize} 
For~$g=3$, the only way for determinant~$+1$ is if the last case appears at least 
once. This implies that the polynomial is divisible by~$(t-i)(t+i) = t^2+1$. By Perron-Frobenius 
theory, we know that~$A$ has at least two real roots. In particular, the first case cannot occur 
and the spectral radius is a totally real algebraic integer with at most two Galois conjugates 
of modulus~$> 1$. If it is not an integer, it is an algebraic integer of degree at least two. 
In particular, the Mahler measure of its minimal polynomial is at least~$\varphi^2$ by Corollary 1' of 
Schinzel~\cite{Schinzel}. Since at most two Galois conjugates have modulus~$>1$, it follows 
that the modulus of the larger root is bounded from below by~$\varphi \approx 1.61 > \mu_6$.

\emph{Case 2: $\det(A) = -1.$} We first rule out the case where all eigenvalues are real. 
The spectral radius is an algebraic integer of degree at most six that is maximal in modulus 
among all its Galois conjugates. Skew-reciprocity of~$A$ and the fact that the minimal polynomial 
has constant coefficient~$\pm 1$ imply that at most half of the Galois conjugates of 
the spectral radius can have modulus~$>1$. Again, Schinzel's Corollary 1' in~\cite{Schinzel} implies 
that the spectral radius is bounded from below by~$\varphi \approx 1.61 > \mu_6$.

In the remaining case, the spectral radius~$\rho$ of~$A$ is of degree six and has four non-real Galois 
conjugates~$\lambda, -\lambda^{-1}, \bar\lambda$ and~$-\bar\lambda^{-1}$. Let the characteristic 
polynomial of~$A$ be given by
\begin{align*}
 P(t) &= t^6  + at^5 + bt^4 + ct^3 - bt^2 + at - 1 \\
 &= (t-\rho)(t + \rho^{-1})(t-\lambda)(t+\lambda^{-1})(t-\bar\lambda)(t+\bar\lambda^{-1}).
\end{align*}

\noindent
We get the following estimates for the coefficients~$a,b$ and~$c$. 

\begin{itemize}
\item Since~$\rho < 1.42$, we have~$|\rho-\rho^{-1}| < 0.72$. For the coefficient~$a$, we get
\begin{align*}
|a| & \le | \rho - \rho^{-1}| + \left|(\lambda + \bar\lambda) - (\lambda^{-1} + \bar\lambda^{-1})\right| \\
 & < 0.72 + 2|\mathrm{Re}(\lambda) - \mathrm{Re}(\lambda^{-1})| 
  =  0.72 + 2|\mathrm{Re}(\lambda) - \frac{\mathrm{Re}(\lambda)}{|\lambda|^2}| \\
  & =  0.72 + 2|\mathrm{Re}(\lambda)| \left(1 - \frac{1}{|\lambda |^{2}}\right) < 0.72 + 1.44 < 3, 
 \end{align*}
where we used~$|\mathrm{Re}(\lambda)| \le |\lambda| < 1.42$ in the second to last inequality.
Up to replacing~$P(t)$ by~$P(-t)$, we may assume that~$a\in\{-2,-1,0\}$. \smallskip

\item Since~$|\lambda| \le \rho < 1.42$, we have~$| \lambda - \lambda^{-1}| < 2.13$ 
and~$| \bar\lambda - \bar\lambda^{-1}| < 2.13$. We calculate 
\begin{align*}
c = -2a + (\rho^{-1} - \rho)(\lambda^{-1} - \lambda)(\bar\lambda^{-1} - \bar\lambda),
\end{align*}
where 
\begin{align*}
\left|(\rho^{-1} - \rho)(\lambda^{-1} - \lambda)(\bar\lambda^{-1} - \bar\lambda)\right| < 0.72 \cdot (2.13)^2 < 3.2.
\end{align*}
In particular,~$c \in \{ -2a-3, \dots , -2a + 3\}.$
\smallskip

\item We have
\begin{align*}
b = -3 + (\rho^{-1} - \rho)(\lambda^{-1} - \lambda + \bar\lambda^{-1} - \bar\lambda)       + (\lambda^{-1} - \lambda)(\bar\lambda^{-1} - \bar\lambda),
\end{align*}
where
\begin{align*}
| (\rho^{-1} - \rho)(\lambda^{-1} - \lambda + \bar\lambda^{-1} - \bar\lambda)   
&+ (\lambda^{-1} - \lambda)(\bar\lambda^{-1} - \bar\lambda) |  \\
    & < 0.72\cdot 1.44 + (2.13)^2 < 5.58. 
\end{align*}
In particular,~$b\in \{-8, \dots, 2\}$. 
\end{itemize}
There are now $3\cdot7\cdot 11 = 231$ remaining polynomials to check.
Listing them all as well as their roots, it is a quick check to see which ones among them have a real root 
with modulus between 1.41 and 1.42; only three polynomials remain. Among these three polynomials, 
only our candidate~$t^6-t^5-t-1$ remains if we insist that the real root with modulus between 1.41 and 1.42 
be maximal in modulus among all the roots of the polynomial.
\end{proof}

\begin{remark}
\emph{
Again we have shown that for~$g\ge3$, the polynomial $$t^{2g}-t^{g+2}-t^{g-2}-1$$ 
is the unique characteristic polynomial that can appear for a matrix minimising the spectral 
radius. The case~$g=3$ is not covered by McMullen's classification but our ad-hoc argument 
rules out all other possibilities for characteristic polynomials: while we gave ourselves the 
liberty to replace~$P(t)$ by~$P(-t)$ during the proof, we note that the root of~$t^6+t^5+t-1$ 
that is maximal in modulus is real and negative. Therefore, the polynomial~$t^6+t^5+t-1$  is not 
the characteristic polynomial of a primitive matrix.
}
\end{remark}

\end{document}